\documentclass[11pt, a4paper]{article}

\usepackage{amsmath}
\usepackage{amssymb}
\usepackage{amsthm}
\usepackage{booktabs}
\usepackage{hyperref}

\theoremstyle{plain}
\newtheorem{theorem}{Theorem}

\newtheorem{lemma}[theorem]{Lemma}
\newtheorem{prop}[theorem]{Proposition}

\newtheorem{fact}[theorem]{Fact}

\theoremstyle{definition}
\newtheorem{conj}{Conjecture}

\newtheorem{question}[conj]{Question}

\title{Prime Nullspace Matrices and Even Dominating Sets}
\author{William F. Klostermeyer\\
   \small School of Computing\\
   \small University of North Florida, Jacksonville, FL 32224-2669\\
   \small E-mail: \href{mailto:wkloster@unf.edu}{wkloster@unf.edu}}
\date{}

\begin{document}

\maketitle

\begin{abstract}
A non-zero $m \times n$ $(0, 1)$-matrix $A$ is called a \emph{nullspace matrix} if each entry $(i, j)$ of $A$ has an even number of $1$'s in the set of entries consisting of $(i, j)$ and its rectilinear neighbors. These matrices correspond to even dominating sets of $m \times n$ grid graphs. We study these matrices and consider those that cannot be constructed from smaller nullspace matrices.
\end{abstract}

\section{Introduction}

A non-zero $m \times n$ $(0, 1)$-matrix $A$ is called a \emph{nullspace matrix} if each entry $(i, j)$ of $A$ has an even number of $1$'s in the set of entries consisting of $(i, j)$ and its rectilinear neighbors. It is called a nullspace matrix since the existence of an $m \times n$ nullspace matrix implies the closed neighborhood matrix of the $m \times n$ grid graph is singular over $\mathrm{GF}(2)$. By \textit{closed neighborhood matrix}, we mean the adjacency matrix of the graph with $1$'s down the diagonal. If this matrix is singular, then the $m \times n$ grid graph contains a non-empty set of vertices $D$ such that each vertex in the grid graph has an even number of elements in its closed neighborhood in $D$, a so-called even dominating set; see for example \cite{gold3}, which contains numerous results on the properties of nullspace matrices and their relationship to even and odd dominating sets in grid graphs.

See Table~\ref{test2} for an example of a nullspace matrix. (All of the arithmetic in this paper is done modulo $2$).

\begin{table}[htbp]
\centering
\caption{A $4 \times 4$ Nullspace Matrix}
\label{test2}
\vspace{2mm}
\begin{tabular}{cccc}
\toprule
1 & 0 & 0 & 0 \\
1 & 1 & 0 & 0 \\
1 & 0 & 1 & 0 \\
0 & 1 & 1 & 1 \\
\bottomrule
\end{tabular}
\end{table}

The row entries in a nullspace matrix can be generated by the Fibonacci polynomials over $\mathrm{GF}(2)$, see \cite{gold3, klost1}. Alternatively, we can compute the entries in row $r_i$ in a nullspace matrix, such as in Table~\ref{test2}, as follows:

\begin{equation}
r_i[j] = r_{i-1}[j] + r_{i-1}[j-1] + r_{i-1}[j+1] + r_{i-2}[j] \pmod 2
\end{equation}

\noindent where undefined entries are taken to be zero. Using this computation, if row $k$ is all zeroes, then rows one through $k-1$ form a nullspace matrix \cite{gold3}.

In terms of the Fibonacci polynomials, if $r_i$ is the $i^{\text{th}}$ row of a nullspace matrix, then

\begin{equation}
r_i = Br_{i-1} + r_{i-2} \quad \text{for } i \ge 2, \quad r_0 = 0, \quad r_1 = w
\end{equation}

\noindent where $B = [b_{ij}]$ is the $n \times n$ tridiagonal $(0, 1)$-matrix with $b_{ij} = 1$ if and only if $|i - j| \le 1$ (and the $r_i$'s in (1) are written as column vectors). If $r_{m+1} = 0$ for some positive integer $m$, then $r_1, r_2, \ldots, r_m$ are the rows of an $m \times n$ nullspace matrix.

Goldwasser and Klostermeyer showed that it follows from the definitions that $r_i = f_i(B)w$ for $i = 0, 1, 2, \ldots$ where $f_i$ is the $i^{\text{th}}$ Fibonacci polynomial over $\mathrm{GF}(2)$ \cite{gold3}:

\begin{equation}
f_i = xf_{i-1} + f_{i-2} \quad \text{for } i \ge 2, \quad f_0 = 0, \quad f_1 = 1
\end{equation}

The following characterizes when nullspace matrices exist.

\begin{theorem}[\cite{gold3}]
There exists an $m \times n$ nullspace matrix if and only if the greatest common divisor of $f_{n+1}(x+1)$ and $f_{m+1}(x)$ is not equal to $1$.
\end{theorem}

As an example, we can use the theorem as follows.

\begin{fact}
There exists a $(2^n-3) \times (2^n-1)$ nullspace matrix for all odd $n \ge 3$.
\end{fact}

\begin{proof}
This is the case if the $\gcd$ of $f_{2^n-2}(x+1)$ and $f_{2^n}(x)$ is not $1$. By Corollary 3 of \cite{gold3}, we immediately have $x \mid f_{2^n}$ for all $n \ge 1$. By Lemma 4 of \cite{gold3}, one can see that $(x+1) \mid f_{2^n-2}$ if $n > 2$ is odd, since $3 \mid (2^n-2)$ when $n$ is odd and since $(x+1) \mid f_3(x)$. Hence the proof.
\end{proof}

One can show by similar methods that $(2^n-3) \times (2^n-1)$ nullspace matrices do not exist when $n > 2$ is even.

We can build large nullspace matrices from smaller ones as shown by the following proposition.

\begin{prop}[\cite{klost1}]
Let $n$ and $k$ be positive integers with $k+1$ a multiple of $n+1$. If there exists an $n \times n$ nullspace matrix, then there also exists a $k \times k$ nullspace matrix.
\end{prop}

That is, if $k+1 = q(n+1)$ where $q$ is a positive integer, and if $A$ is an $n \times n$ nullspace matrix, then a $k \times k$ nullspace matrix can be constructed by letting row and column numbers $n+1, 2(n+1), \ldots, (q-1)(n+1)$ have all entries equal to zero, creating a $q \times q$ array of $n \times n$ squares, putting $A$ in one of the $n \times n$ squares and filling in the rest of them by ``reflecting'' across the lines of zeroes. See Table~\ref{null9} for an example.

\begin{table}[htbp]
\centering
\caption{A $9 \times 9$ Nullspace Matrix}
\label{null9}
\vspace{2mm}
\begin{tabular}{ccccccccc}
\toprule
1 & 0 & 0 & 0 & 0 & 0 & 0 & 0 & 1 \\
1 & 1 & 0 & 0 & 0 & 0 & 0 & 1 & 1 \\
1 & 0 & 1 & 0 & 0 & 0 & 1 & 0 & 1 \\
0 & 1 & 1 & 1 & 0 & 1 & 1 & 1 & 0 \\
0 & 0 & 0 & 0 & 0 & 0 & 0 & 0 & 0 \\
0 & 1 & 1 & 1 & 0 & 1 & 1 & 1 & 0 \\
1 & 0 & 1 & 0 & 0 & 0 & 1 & 0 & 1 \\
1 & 1 & 0 & 0 & 0 & 0 & 0 & 1 & 1 \\
1 & 0 & 0 & 0 & 0 & 0 & 0 & 0 & 1 \\
\bottomrule
\end{tabular}
\end{table}

Of course, letting $x=1$ in the Fibonacci polynomials generates the Fibonacci numbers. We are inspired by the famous conjecture that there are infinitely many primes that are also Fibonacci numbers. This leads us to explore the concept of ``prime'' nullspace matrices in the next section.

\section{Nullspace-Primes}

A nullspace matrix that has at least one row or column of zeroes is a \textit{composite nullspace matrix}, otherwise it is a \textit{prime nullspace matrix}. An integer $n$ is \textit{nullspace-prime} if there exists an $(n-1) \times (n-1)$ nullspace matrix, but for no proper divisor $m$ of $n$ does there exist an $(m-1) \times (m-1)$ nullspace matrix. From Table~\ref{null9}, we see that $10$ is not nullspace-prime. In \cite{klost1}, it is reported that the first few nullspace-primes are 5, 6, 17, 31, 33, 63, 127, 129, 171, 257, 511, 683. It is easy to see that there exists an $n \times n$ nullspace matrix if and only if $n$ is one less than a multiple of a nullspace-prime.

It was conjectured in \cite{klost1} that there are infinitely many nullspace primes and this was proved in \cite{hunz}. We present a slightly different proof. A lemma is needed first.

\begin{lemma}\label{rows}
If $n \times n$ matrix $M$ is composite, then every $i^{\text{th}}$ row must be all-zeroes, for some integer $i > 0$.
\end{lemma}

\begin{proof}
Assume without loss of generality that $M$ has a row of zeroes (else rotate $M$ ninety degrees). Let $k > 0$ be the smallest integer such that row $k$ of $M$ is all-zeroes. By equation (1), rows $k-1$ and $k+1$ will be equal. Since equation (1) can be applied in both top-down and bottom-up directions, it follows that row $k-2$ will equal row $k+2$ and so on. The lemma follows.
\end{proof}

\begin{theorem}
There are an infinite number of nullspace-primes.
\end{theorem}

\begin{proof}
We know from \cite{gold3} that if $n=2^k$, where $k > 3$, or $n=2^k-2$ where $k > 3$, then there exists an $n \times n$ nullspace matrix.

Suppose there were only finitely many nullspace-primes. Then there is a largest nullspace prime $q$. Then we can choose a prime number $p$ greater than $q$. Since all prime divisors of $2^p-1$ are congruent with $1 \pmod p$, all non-trivial divisors of $2^p-1$ are greater than $q$. However, by Lemma~\ref{rows}, one of the divisors of $2^p-1$ must be nullspace-prime, since $2^p-1$ is not nullspace-prime. This is a contradiction.
\end{proof}

Define $n$ to be \textit{super nullspace-prime} if there exists no $(n-1) \times (n-1)$ composite nullspace matrix and there exists an $(n-1) \times (n-1)$ nullspace matrix. As stated above, $63$ is a nullspace-prime, so there is no way to ``piece together'' square nullspace matrices to get a $62 \times 62$ nullspace matrix. But there exists a $6 \times 8$ nullspace matrix. A $9 \times 7$ array of this nullspace matrix and its reflections, with rows and columns of zeroes in between, can be used to construct a composite $62 \times 62$ nullspace matrix. Hence $63$ is nullspace-prime, but not super nullspace-prime.

It is easy to see that the rows of zeroes in an $n \times n$ nullspace matrix occur in regular intervals and likewise the columns (though these may be two different intervals, as in the $62 \times 62$ case). To see this, note that equation (1) is symmetric about rows (columns) of zeroes: if row (column) $n$ is all-zeroes, then rows (columns) $n-1$ and $n+1$ will be equal, for $n \ge 1$.

\begin{conj}\label{nullsp}
There are an infinite number of super nullspace-primes.
\end{conj}

Because there exists a $(2^n-2) \times (2^n-2)$ nullspace matrix for all $n > 3$ \cite{gold3}, Conjecture~\ref{nullsp} is true if there are infinitely many Mersenne primes, the existence of which is one of the most famous conjectures of number theory. Though $2^n-1$ is prime only if $n$ is prime, it is possible for $2^n-1$ to be nullspace prime for composite $n$ such as $n=9$.

By the same token, because there exists a $2^n \times 2^n$ nullspace matrix for all $n > 3$, see \cite{gold3}, Conjecture~\ref{nullsp} is true if there are infinitely many Fermat primes (primes of the form $2^n+1$, of which only five are currently known: 3, 5, 17, 257, and 65,537). Of course, we have nullspace primes of the form $2^n+1$ such as 33 and 129 that are not Fermat primes. However, the next smallest nullspace-prime of the form $2^n+1$ is 65,537. For $2^n+1$ to be prime, it is well-known that $n=2^k$, for some $k \ge 0$.

\begin{fact}
Let $q=2^n+1$ be such that $q \ge 5$ and $n \neq 2^k$, for any $k \ge 1$. If $3$ does not divide $q$, then $q$ is not nullspace-prime.
\end{fact}

\begin{proof}
It is well-known that $q$ has a factor of the form $2^a + 1$. Since $3$ does not divide $q$, there is a nullspace matrix of size $2^a \times 2^a$. Hence the result.
\end{proof}

Though 513, 2049, 8193 are each divisible by 3, they are not nullspace-prime due to their other factors (e.g., 171 divides 513 and there exists a $170 \times 170$ nullspace matrix).  Not all integers of this form are nullspace prime, 32, 679 being an example as  there is no nullspace matrix of dimension $32,678 \times 32,678$
%On the other hand, 9 divides 32,679, but it is nullspace-prime (since there is no $3630 \times 3630$ nullspace matrix). We have not done enough computation to determine if 32,679 is super nullspace-prime or not.

Conjecture~\ref{nullsp} is also true if there exist infinitely many primes $p$ for which there exist $(p-1) \times (p-1)$ nullspace matrices. In fact, this condition can be weakened by saying there exist infinitely many primes $p$ such that there exists a $p^{k-1} \times p^{k-1}$ nullspace matrix for some $k > 0$.

Let $f_n$ be the $n^{\text{th}}$ Fibonacci polynomial (again, all arithmetic is over $\mathrm{GF}(2)$). In terms of the polynomials, the following is from \cite{gold3}.

\begin{theorem}
Let $n$ be an even positive integer and let $d_{n+1}(x)$ have positive degree and be the greatest common divisor of $f_{n+1}(x)$ and $f_{n+1}(x+1)$. Then every $n \times n$ nullspace matrix is prime (i.e., has no row/column of zeroes) if and only if every irreducible factor of $d_{n+1}(x)$ is a factor of $f_{n+1}(x)$ and is a factor of no smaller Fibonacci polynomial.
\end{theorem}

It may be that further analysis along these lines may be able to resolve Conjecture~\ref{nullsp}.

A lemma is needed before the next result. The following lemma is from \cite{gold3}.

\begin{lemma}[\cite{gold3}]\label{lots}
Let $k$ be a positive integer. Then $f_{2^k+1}(x) \cdot f_{2^k-1}(x)$ is equal to the square of the product of all irreducible polynomials with degrees dividing $k$, except $x$.
\end{lemma}

That is, $f_{2^n-1}(x)$ and $f_{2^n+1}(x)$ have, informally speaking, ``lots'' of factors. We say a polynomial $f(x)$ is \textit{self-conjugate} if $f(x)=f(x+1)$.

\begin{theorem}\label{super2}
The integer $2^{2n}-1$ is not super nullspace-prime, for all $n \ge 1$.
\end{theorem}

\begin{proof}
From Theorem 14 of \cite{gold3}, there exists a $(2^{k}-2) \times (2^{k}-2)$ nullspace matrix, for all $k > 3$. It is easy to see that 3 is not super nullspace-prime. So we assume $n > 1$.

Observe that $2^{2n}-1=(2^n-1)(2^n+1)$. So it suffices to prove that there exists a $(2^n-2) \times 2^n$ nullspace matrix, for all odd $n > 2$. This is the case if the $\gcd$ of $f_{2^n-1}(x+1)$ and $f_{2^n+1}(x)$ is not equal to 1, when $n$ is odd. So the theorem is true for any particular value of $n$ if some factor of $f_{2^n-1}(x)$ is also a factor of $f_{2^n+1}(x+1)$, or vice versa.

From \cite{gold3}, we know that $f_{2^n+1} + f_{2^n-1} = x^{2^n}$, since $f_{2^n}=x^{2^n-1}$, for all $n \ge 1$. By Lemma 2(d) of \cite{gold3} we know that $f_{2n+1}=f_{n}^{2}+f_{n+1}^{2}$, which allows us to determine the value of $f_{2^n-1}(x)$ precisely using a simple induction. Because $f_{2^n}(x)=x^{2^n-1}$, we can deduce that $f_{2^n+1}(x) = x^{2^n} + f_{2^n-1}(x) = x^{2^n} + x^{2^n-2} + \ldots + x^{2^{n-1}} + 1$.

We claim that neither $f_{2^n-1}(x)$ nor $f_{2^n+1}(x)$ are self-conjugate. By Lemma~\ref{lots}, it suffices to show that one of them is not self-conjugate. Note that one of $f_{2^n-1}(x)$, $f_{2^n+1}(x)$ has an even number of terms and one has an odd number of terms. Therefore when we substitute $(x+1)$ for $x$, one of them will have the coefficient of $x^0$ equal to 1 and one will have the coefficient of $x^0$ equal to 0. Hence it cannot be that both are self-conjugate.
\end{proof}

Generalizing the properties of the self-conjugate polynomials in the previous result, we believe the following to be true.

\begin{conj}
Let $p(x)$ be an even polynomial over $\mathrm{GF}(2)$. Then $p(x)$ is self-conjugate if and only if $p(x)=x^{2^{k_1}} + x^{2^{k_2}} + \ldots + x^{2^{k_m}}$, where $m$ is odd if $x^0$ is a term and $m$ is even otherwise.
\end{conj}

\begin{question}
Is the following true: Nullspace-prime $n$ is not super nullspace-prime if and only if $n=2^{2k}-1$ for some $k \ge 1$.
\end{question}

We might ask the more general question of whether there are infinitely many nullspace primes that are not super nullspace-primes.

We leave as an open problem whether or not there exists an efficient algorithm to determine if an integer is nullspace-prime or super nullspace-prime.

\subsection*{Acknowledgments}

We thank John Goldwasser for his valuable suggestions.

\section{Updates}

Conjecture 2 has been shown false, e.g., $p(x) = (x^4 + x^2)^3$. Question 3 has been answered in the negative. The integers less than 10,000 that were found to be nullspace prime but not super nullspace prime are  6, 63, 129, 171, 511, 2047, 3277, 3641 . The first seven super nullspace primes are 5, 17, 31, 33, 127, 257, 683. The nulspace primes less than 10,000 are 5, 6, 17, 31, 33, 63, 127, 129, 171, 257, 511, 683, 2047, 2731, 2979, 3277, 3641, 8191.

These were found by Haobo Ma and Wenlin Zhang.

\end{document}